 \documentclass[twoside,reqno,11pt]{fcaa-var} %

\usepackage{graphicx}
\usepackage{epsfig}
\usepackage{amsthm}
\usepackage{amsmath}
\usepackage{latexsym}
\usepackage{amsfonts}
\usepackage{amssymb}

\newtheoremstyle{theorem}
  {15pt}          
  {15pt}  
  {\sl}  
  {\parindent}
  {\sc}  
  {. }   
  { }    
  {}     
\theoremstyle{theorem}
\newtheorem{lemma}{Lemma}[section]
\newtheorem{theorem}{Theorem}[section]
\newtheorem{corollary}{Corollary}[section]

\newtheoremstyle{defi}
  {15pt}          
  {15pt}  
  {\rm}  
  {\parindent}     
  {\sc}  
  {. }    
  { }    
  {}     
\theoremstyle{defi}

\newtheorem{remark}{Remark}[section]

 \def\proofname{\indent\rm P\ r\ o\ o\ f.}
 
 \def\qed{\hfill$\Box$} 

 \usepackage{hyperref} 
\newcommand{\ds}{\displaystyle}
\newcommand{\f}{\frac}

\usepackage{amssymb}
\usepackage{latexsym}
\usepackage{amsmath}
\allowdisplaybreaks

 \def\theequation{\arabic{section}.\arabic{equation}}

 \def\themycopyrightfootnote{\vspace*{3pt}
 \copyright 2016\, Year\,  Diogenes  Co., Sofia
   \par  \noindent pp. 1--29, DOI: ......................
   \hfill  \vspace*{-36pt}
  }

 \title[Neumann boundary-value problem of FDE]{Wellposedness of Neumann boundary-value problems of space-fractional differential equations}
 \author[\normalsize H. Wang and D.P. Yang]{\normalsize Hong Wang $^1$ \qquad Danping Yang $^2$}

\begin{document}

 \vbox to 2.5cm { \vfill }


 \bigskip \medskip

\begin{abstract}
Fractional differential equation (FDE) provides an accurate description of transport processes that exhibit anomalous diffusion but introduces new mathematical difficulties that have not been encountered in the context of integer-order differential equation. For example, the wellposedness of the Dirichlet boundary-value problem of one-dimensional variable-coefficient FDE is not fully resolved yet. In addition, Neumann boundary-value problem of FDE poses significant challenges, partly due to the fact that different forms of FDE and different types of Neumann boundary condition have been proposed in the literature depending on different applications. 

We conduct preliminary mathematical analysis of the wellposedness of different Neumann boundary-value problems of the FDEs. We prove that five out of the nine combinations of three different forms of FDEs that are closed by three types of Neumann boundary conditions are well posed and the remaining four do not admit a solution. In particular, for each form of the FDE there is at least one type of Neumann boundary condition such that the corresponding boundary-value problem is well posed, but there is also at least one type of Neumann boundary condition such that the corresponding boundary-value problem is ill posed. This fully demonstrates the subtlety of the study of FDE.

 \medskip


 \smallskip


 \end{abstract}

 \maketitle

 \vspace*{-16pt}


\section{Introduction}\label{sec:1}

\setcounter{section}{1}
\setcounter{equation}{0}\setcounter{theorem}{0}

FDE is emerging as a powerful and competitive tool for modeling challenging phenomena such as anomalous diffusion as well as long-range time memory and spatial interactions, which cannot be modeled accurately by integer-order differential equation \cite{BurHal,CheLia,DelCar,Hil,Mag,MaiRab,MeeSik,MetKla04,OldSpa,Pod}. However, FDE introduces new mathematical difficulties that have not been encountered in the context of integer-order differential equation. The wellposedness of the boundary-value problem of FDE and the regularity of its solution are representative issues in the study of FDE.

A Galerkin weak formulation was derived for the homogeneous Dirichlet boundary-value problem of a one-dimensional conservative Caputo FDE of order $2-\beta$ for $0 < \beta < 1$ with a constant diffusivity coefficient $K$. The formulation was proved to be coercive and bounded on the fractional Sobolev space $H^{1-\beta/2}_0$ \cite{AdaFou}, which ensures its wellposedness in $H^{1-\beta/2}_0$ \cite{ErvRoo}. However, the Galerkin formulation loses its coercivity for a variable coefficient $K$ and the corresponding Galerkin finite element method (FEM) may diverge \cite{WanYan,WanYanZhu14}. A Petrov-Galerkin weak formulation was derived for a one-sided variable-coefficient FDE with proved weak coercivity, which ensures the wellposedness of the weak formulation \cite{WanYan}. A Petrov-Galerkin FEM was derived accordingly with a proved error estimate \cite{WanYanZhu15}. It is well known that for the homogeneous Dirichlet boundary condition, the conservative Riemann-Liouville FDE and conservative Caputo FDE coincide \cite{ErvRoo,WanYan}. However, it was shown in \cite{WanYanZhu14} that the inhomogeneous Dirichlet boundary-value problem of conservative Caputo FDE is well posed while that of the conservative Riemann-Liouville FDE does not admit a weak solution! 

A closely related issue is the regularity of the solutions of FDE. It was shown in \cite{WanYanZhu14,WanZha} that the true solution to the homogeneous Dirichlet boundary-value problem of a one-dimensional linear diffusion FDE of order $2-\beta$ with a constant diffusivity coefficient and a constant source term is not in the fractional Sobolev space $W^{1,1/\beta}$ \cite{AdaFou}. In particular, the true solution is not in the Sobolev space $H^1$ for $0 < \beta < 1/2$. Consequently, the optimal-order convergence rate in the energy norm, and, in particular, the Nitsche-lifting based proof of optimal-order $L^2$ error estimate are not valid \cite{ErvRoo}, as the regularity assumption of the true solution is not satisfied. Numerical experiments justify the observation \cite{WanYanZhu14,WanZha}. This is in sharp contrast to the regularity of solution to integer-order linear elliptic differential equation, which can be ensured by the smoothness of the data of the differential equation and of the boundary \cite{GilTru}. To date, there is no verifiable condition in the literature which can ensure the regularity of the solution to FDE. A thorough regularity analysis was present in \cite{JinLazPas} for the homogeneous Dirichlet boundary-value problem of a one-sided constant-coefficient space-fractional PDE in one space dimension, as the true solution can be found analytically in a closed form.
In short, the wellposedness of the homogeneous Dirichlet boundary-value problem of linear variable-coefficient FDE in one space dimension has not been resolved completely yet.

The Neumann boundary-value problem of FDE poses even more challenges than the Dirichlet boundary-value problem, partly due to the fact that different types of Neumann boundary conditions were proposed in the literature depending on different applications. This is in addition to different forms of FDE that have already demonstrated significantly different mathematical properties \cite{WanYanZhu14}. In this paper we conduct preliminary mathematical analysis of the wellposedness of different Neumann boundary-value problems of the different forms of FDEs. We prove that five out of the nine combinations of three different forms of FDEs that are closed by three types of Neumann boundary conditions are well posed and the remaining four do not admit a solution. In particular, for each form of the FDE there is at least one type of Neumann boundary condition such that the corresponding boundary-value problem is well posed, but there is also at least one type of Neumann boundary condition such that the corresponding boundary-value problem is ill posed. This fully demonstrates the subtlety of the study of FDE.

\section{Problem Formulation}
\setcounter{section}{2}\setcounter{equation}{0}

Let $C^m[0,1]$ be the space of continuously differentiable functions of orders up to $m$ on $[0,1]$. Let $C^{m,\delta}[0,1]$ be the space of H\"{o}lder continuous functions of order $m$ on $[0,1]$. Let $C^\infty_0(0,1)$ be the space of infinitely many times differentiable functions on $(0,1)$ that are compactly supported within $(0,1)$. Let $L^p(0,1)$, with $1 \le p \le +\infty$, be the normed space of $p$-th power Lebesgue integrable functions on $(0,1)$. Let $W^{m,p}(0,1)$ be the Sobolev space of functions on $(0,1)$ whose weak derivatives up to order $m$ are in $L^p(0,1)$. Let $H^\mu(0,1)$, with $\mu > 1/2$, be the fractional Sobolev space of order $\mu$ and $H^\mu_0(0,1)$ be the completion of $C^\infty_0(0,1)$ with respect to the Sobolev norm $\| \cdot \|_{H^\mu(0,1)}$. Let $H^{-\mu}(0,1)$ be the dual space of $H^\mu_0(0,1)$ \cite{AdaFou}.

Let $Dw := w^\prime(x)$ be the first-order differential operator. For $0 < \alpha < 1$ the left and right fractional integrals of order $\alpha$ are defined for any $w \in C[0,1]$ by \cite{Pod,SamKil}
$$\begin{array}{rl}
\ds {}_0I_x^{\alpha}w(x) := \f1{\Gamma(\alpha)} \int_0^x \f{w(s)}{(x-s)^{1-\alpha}}ds,\quad
{}_xI_1^{\alpha}w(x) := \f1{\Gamma(\alpha)} \int_x^1 \f{w(s)}{(s-x)^{1-\alpha}}ds
\end{array}$$
where $\Gamma(\cdot)$ is the Gamma function. It is clear that 
\begin{equation}\label{FDE:e0}\begin{array}{ll}
{}_0I^{\alpha}_x x^\mu &\ds = \frac{\Gamma(\mu+1)}{\Gamma(\alpha+\mu+1)}x^{\alpha+\mu}, \\[0.1in]
{}_xI^{\alpha}_1 (1-x)^\mu &\ds = \frac{\Gamma(\mu+1)}{\Gamma(\alpha+\mu+1)}(1-x)^{\alpha+\mu}
\end{array}\end{equation}
holds for any $0 < \alpha <1$ and $\mu >-1$. 

For any positive integer $m$ and $0 < \alpha < 1$, the left and right Caputo fractional derivatives of order $m-\alpha$ are defined by \cite{Pod,SamKil}
$${}^C_0D_x^{m-\alpha}w(x) := {}_0I_x^{\alpha} D^m w(x), \quad {}^C_xD_1^{m-\alpha}w(x) := {}_xI_1^{\alpha} (-D)^m w(x), 
$$
and the left and right Riemann-Liouville fractional derivatives of order $m-\alpha$ are 
$${}^R_0D_x^{m-\alpha}w(x) := D^m {}_0I_x^{\alpha} w(x), \quad {}^R_xD^{m-\alpha}_1w(x) := (-D)^m {}_xI_1^{\alpha} w(x).$$

The one-sided Caputo FDE of order $2-\beta$ is 
\begin{equation}\label{FDE:Caputo}
- {}^C_0D_x^{2-\beta}w(x) = f(x), \quad  x \in (0,1).
\end{equation}
The conservative Caputo FDE of order $2-\beta$ is 
\begin{equation}\label{FDE:Conserv}
 - D \bigl ({}^{C}_0D_x^{1-\beta}w(x) \bigr) = f(x), \quad  x \in (0,1).
\end{equation}
The Riemann-Liouville FDE of order $2-\beta$ is of the form 
\begin{equation}\label{FDE:RL}
 - {}^{R}_0D_x^{2-\beta}w(x) = f(x), \quad  x \in (0,1).
\end{equation}

\begin{remark}
For the homogeneous Dirichlet boundary condition, the conservative Caputo FDE \eqref{FDE:Conserv} and the Riemann--Liouville FDE \eqref{FDE:RL} coincide \cite{ErvRoo,WanYan,WanYanZhu14}. But these two equations differ in the current context.
\end{remark}

We consider the classical Neumann boundary condition
\begin{equation}\label{FDE:Nbc}
 Du \big |_{x=0} = a_0, \quad Du \big |_{x=1} = a_1,
\end{equation}
or the Caputo fractional Neumann boundary condition
\begin{equation}\label{FDE:Cbc}
{}^{C}_0D_x^{1-\beta}u \bigl |_{x=0} = a_0, \quad {}^{C}_xD_1^{1-\beta}u \big |_{x=1} = a_1,
\end{equation}
or the Riemann--Liouville fractional Neumann boundary condition
\begin{equation}\label{FDE:Rbc}
{}^{R}_0D_x^{1-\beta}u |_{x=0} = a_0, \quad {}^{R}_xD_1^{1-\beta}u \bigr|_{x=1} = a_1.
\end{equation}

We next cite some known results that are to be used in this paper \cite{AdaFou,Pod,SamKil,WanYan}.
\begin{lemma}\label{lem:Aux1} The left and right Riemann--Liouville fractional integral operators
are adjoints in the $L^2$ sense, i.e., for all $\alpha  > 0$,
\begin{equation}\label{Aux1:e1}
\big({}_0I^{\alpha}_x w, v \big)_{L^2(0,1)} = \big(w, {}_xI^{\alpha}_1 v \big)_{L^2(0,1)}
\qquad \forall w, ~v \in L^2(0,1).
\end{equation}
The left and right Riemann--Liouville fractional integral operators follow the properties of
a semigroup, i.e., for any $w \in L^p(0,1)$ and $\alpha, \mu > 0$
\begin{equation}\label{Aux1:e2}
{}_0I^{\alpha}_x {}_0I^{\mu}_x w(x) = {}_0I^{\alpha+\mu}_x w(x), ~~
{}_xI^{\alpha}_1 {}_xI^{\mu}_1 w(x) = {}_x I^{\alpha+\mu}_1 w(x), ~~\forall x \in [0,1].
\end{equation}
\end{lemma}
\begin{lemma}\label{lem:Aux2}
Let $\alpha > 0$ and $0 < \mu < 1$. Then for any $w \in L^p(0,1)$
\begin{equation}\label{Aux2:e1}
{}^R{}_0 D^{\alpha}_x ~{}_0 I^{\alpha}_x w(x) = w(x), \quad {}^R_xD_1^{\alpha} ~{}_xI^{\alpha}_1 w(x) = w(x).
\end{equation}
For any $v, w \in C^1[0,1]$ with $v(0)=0$ and $w(1)=0$, the following equalities hold
\begin{equation}\label{Aux2:e2}\begin{array}{ll}
{}^R_0 D^{\mu}_x v(x) = {}^C_0 D^{\mu}_x v(x), ~~ &{}_0I^{\mu}_x~ {}^R_0D^{\mu}_x v(x) = {}_0I^{\mu}_x ~{}^C_0D^{\mu}_x v(x) = v(x), \\[0.1in]
{}^R_xD^{\mu}_1 w(x) = {}^C_x D^{\mu}_1 w(x), &{}_xI^{\mu}_1 ~{}^R_xD^{\mu}_1 w(x) = {}_xI^{\mu}_1 ~{}^C_xD^{\mu}_1 w(x) = w(x).
\end{array}\end{equation}
\end{lemma}

\begin{lemma}\label{lem:Aux0} Let $\alpha > 0$. 
\begin{itemize}
\item[{\rm (i)}] ${}_0I^{\alpha}_x$: $L^2(0,1) \rightarrow L^2(0,1)$ is a bounded linear operator.

\item[{\rm (ii)}] ${}_xI^{\alpha}_1$: $L^2(0,1) \rightarrow L^2(0,1)$ is a bounded linear operator.
\end{itemize}
\end{lemma}

\begin{lemma}\label{lem:Fred} {\rm (Friedrichs inequality)} 
\begin{equation*}
\| v \|_{L^2(0,1)} \le C \Bigl ( \Bigl | \int^1_0 vdx \Bigr | + \|Dv \|_{L^2(0,1)} \Bigr ), \qquad  \forall v \in H^1(0,1).
\end{equation*}
Consequently, $\| v \|_{H^{1,0}(0,1)} :=  \|Dv \|_{L^2(0,1)}$ defines a norm on the space $H^{1,0}(0,1)$
\[ H^{1,0}(0,1) := \Bigl \{ v \in H^1(0,1): ~ \int_0^1 v dx = 0 \Bigr \}. \]
\end{lemma}

In this paper we also use $C$ with or without subscripts to denote generic positive constants that may assume different values at different occurences.

\section{Nonconventional fractional derivative spaces}
\setcounter{section}{3}\setcounter{equation}{0}

We introduce some nonconventional fractional derivative spaces and study their properties. For $0 < \beta <1$ and $\varepsilon = \varepsilon(\beta)$ with $0 < \varepsilon(\beta) <<1$ and $\varepsilon(\beta) < 1 - \beta$, we define $\kappa = \kappa(\beta)$ to be
\begin{equation*}
\kappa(\beta): =\left\{ \begin{array}{ll}
2, \ \   & 0 < \beta <1/2;   \\[0.05in]
1+(1 - \beta - \varepsilon(\beta))/\beta, & 1/2 \leq \beta <1.
\end{array}\right.
\end{equation*}
In particular, $1 < \kappa(\beta) < 1/\beta$ but sufficiently close to $1/\beta$ for $1/2 \le \beta < 1$.

For $0 < \mu < 1$ we define left and right Riemann--Liouville fractional derivative spaces
\begin{equation}\label{FDS:e1}\begin{array}{l}
H^\mu_{R,l}(0,1) := \bigl \{ v \in L^\kappa(0,1): ~~ {}^R_0D_x^\mu v \in L^2(0,1) \bigr \}, \\[0.1in]
H^\mu_{R,r}(0,1) := \bigl \{ v \in L^\kappa(0,1): ~~ {}^R_xD_1^\mu v \in L^2(0,1) \bigr \}
\end{array}\end{equation}
equipped with the (semi) norms
\begin{equation*}\begin{array}{lrl}
|v|_{H^\mu_{R,l}(0,1)} := \| {}^R_0D_x^\mu v\|^2_{L^2(0,1)}, & \|v\|_{H^\mu_{R,l}(0,1)} := \big (\|v\|^2_{L^\kappa(0,1)} + |v|_{H^\mu_{R,l}(0,1)}^2 \big)^\frac{1}{2}, \\[0.1in]
|v|_{H^\mu_{R,r}(0,1)} := \| {}^R_xD_1^\mu v\|^2_{L^2(0,1)}, & \|v\|_{H^\mu_{R,r}(0,1)} := \big (\|v\|^2_{L^\kappa(0,1)} + |v|_{H^\mu_{R,r}(0,1)}^2 \big)^\frac{1}{2}.
\end{array}\end{equation*}
We also define subspaces $H^{\mu,0}_{R,l}(0,1) \subset H^\mu_{R,l}(0,1)$ and $H^{\mu,0}_{R,r}(0,1) \subset H^{\mu}_{R,r}(0,1)$
\begin{equation}\label{FDS:e2}\begin{array}{rl}
H^{\mu,0}_{R,l}(0,1) &\ds := \left \{ v \in H^{\mu}_{R,l}(0,1): ~~ \int^1_0 {}_0I_x^{1-\mu}v(x) dx=0 \right \}, \\[0.1in]
H^{\mu,0}_{R,r}(0,1) &\ds := \left \{ v \in H^{\mu}_{R,r}(0,1): ~~ \int^1_0 {}_xI_1^{1-\mu}v(x) dx=0 \right \}.
\end{array}\end{equation}
Corollary \ref{cor:RFred} shows that
\begin{equation*}
\|v\|_{H^{\mu,0}_{R,l}} := |v|_{H^{\mu}_{R,l}(0,1)}, \quad \|v\|_{H^{\mu,0}_{R,r}} := |v|_{H^{\mu}_{R,r}(0,1)}
\end{equation*}
define a norm on the spaces $H^{\mu,0}_{R,l}(0,1)$ and $H^{\mu,0}_{R,r}(0,1)$, respectively.

We similarly define left and right Caputo fractional derivative spaces
\begin{equation}\label{FDS:e3}\begin{array}{rl}
H^{\mu}_{C,l}(0,1) &:= \bigl \{ v \in L^\kappa(0,1): ~~ {}^C_0 D_x^{\mu}v \in L^2(0,1) \bigr \}, \\[0.1in]
H^{\mu}_{C,r}(0,1) &:= \bigl \{ v \in L^\kappa(0,1): ~~ {}^C_x D_1^{\mu}v \in L^2(0,1) \bigr \}
\end{array}\end{equation}
equipped with the (semi) norms
\begin{equation*}\begin{array}{rlrl}
|v|_{H^{\mu}_{C,l}(0,1)} &:= \| {}^C_0 D_x^{\mu} v \|^2_{L^2(0,1)}, & \|v\|_{H^{\mu}_{C,l}(0,1)} := \big (\|v\|^2_{L^\kappa(0,1)} + |v|_{H^{\mu}_{C,l}(0,1)}^2 \big)^\frac{1}{2}, \\[0.1in]
|v|_{H^{\mu}_{C,r}(0,1)} &:= \| {}^C_xD_1^{\mu} v\|^2_{L^2(0,1)}, & \|v\|_{H^{\mu}_{C,r}(0,1)} := \big (\|v\|^2_{L^\kappa(0,1)} + |v|_{H^{\mu}_{C,r}(0,1)}^2 \big)^\frac{1}{2}.
\end{array}\end{equation*}
We also define subspaces $H^{\mu,0}_{C,l}(0,1) \subset H^{\mu}_{C,l}(0,1)$ and $H^{\mu,0}_{C,r}(0,1) \subset H^{\mu}_{C,r}(0,1)$
\begin{equation}\label{FDS:e4}\begin{array}{rl}
H^{\mu,0}_{C,l}(0,1) &\ds := \left \{ v \in H^{\mu}_{C,l}(0,1): ~~ \int^1_0 v dx=0 \right \}, \\[0.1in]
H^{\mu,0}_{C,r}(0,1) &\ds := \left \{ v \in H^{\mu}_{C,r}(0,1): ~~ \int^1_0 v dx=0 \right \}.
\end{array}\end{equation}
Lemma \ref{lem:CFred} shows that
\begin{equation*}
\|v\|_{H^{\mu,0}_{C,l}} := |v|_{H^{\mu}_{C,l}(0,1)}, \quad \|v\|_{H^{\mu,0}_{C,r}} := |v|_{H^{\mu}_{C,r}(0,1)}
\end{equation*}
define a norm on the spaces $H^{\mu,0}_{C,l}(0,1)$ and $H^{\mu,0}_{C,r}(0,1)$, respectively .

For $0 < \beta < 1/2$, $\kappa(\beta) = 2$ and so $L^\kappa(0,1) = L^2(0,1)$. The left and right Riemann--Liouville fractional derivative spaces \eqref{FDS:e1} coincide with those introduced in \cite{ErvRoo}. For the homogeneous Dirichlet boundary condition, the Riemann--Liouville fractional derivative spaces \eqref{FDS:e1} and the Caputo fractional derivative spaces \eqref{FDS:e3} equal to the fractional Sobolev space $H^{1-\beta}_0(0,1)$ with equivalent norms.

However, without the homogeneous Dirichlet boundary condition, the Riemann-Liouville and Caputo fractional derivative spaces differ from each other. For example, choosing $\alpha = \beta$ and $\mu = -\beta$ in \eqref{FDE:e0} yields
\begin{equation}\label{FDS:e5}
{}_0I^{\beta}_x x^{-\beta} = \Gamma(1-\beta), \qquad {}^R_0 D^{2-\beta}_x x^{-\beta} = {}^R_0 D^{1-\beta}_x x^{-\beta} = 0.
\end{equation}
Hence, $x^{-\beta}$ is a solution to the homogeneous Riemann--Liouville FDE \eqref{FDE:RL} and ${}^R_0 D^{1-\beta}_x x^{-\beta} \in L^2(0,1)$ for any $0 < \beta < 1$. Since $x^{-\beta} \in L^2(0,1)$ for $0 < \beta < 1/2$, $x^{-\beta} \in H^{1-\beta}_{R,l}(0,1)$ in this case. However,
\begin{equation}\label{FDS:e6}
{}^C_0 D^{1-\beta}_x x^{-\beta} = - \beta {}_0I^{\beta}_x x^{-\beta-1} = -\infty, \qquad 0 < x < 1.
\end{equation}
Hence, $x^{-\beta} \notin H^{1-\beta}_{C,l}(0,1)$ for any $0 < \beta < 1$.

For $1/2 \le \beta < 1$, \eqref{FDS:e5} shows that $x^{-\beta}$ is still a solution to the homogeneous Riemann--Liouville FDE \eqref{FDE:RL} and ${}^R_0 D^{1-\beta}_x x^{-\beta} \in L^2(0,1)$, but $x^{-\beta} \notin L^2(0,1)$. If we insist that $H^{1-\beta}_{R,l}(0,1) \subset L^2(0,1)$ as those in \cite{ErvRoo}, a weak solution might not exist in these spaces. Note that $x^{-\beta} \in L^\kappa(0,1)$ and ${}^R_0 D^{1-\beta}_x x^{-\beta} \in L^2(0,1)$, so $x^{-\beta} \in H^{1-\beta}_{R,l}(0,1)$ in \eqref{FDS:e1}. This observation motivates the introduction of the fractional derivative spaces $H^{1-\beta}_{R,l}(0,1)$ and $H^{1-\beta}_{R,r}(0,1)$ in \eqref{FDS:e1} which could accommodate solutions with singularity $x^{-\beta}$ or $(1-x)^{-\beta}$, respectively. But for $1/2 \le \beta < 1$ these spaces are nonconventional and not Hilbert spaces anymore.

\begin{theorem}\label{thm:iso} Let $0 < \beta < 1$. The fractional integral operator $I^{\beta}_+$ (or $I^{\beta}_-$) defines an isomorphism from $H^{1-\beta}_{R,l}(0,1)$ (or $H^{1-\beta}_{R,r}(0,1)$) onto $H^1(0,1)$ with  equivalent norms, i.e., there exist positive constants $0 < C_0 \leq C_1 < +\infty$ such that
$$\begin{array}{c}
C_0 \big \| I_+^{\beta} v \big \|_{H^1(0,1)} \leq \|v\|_{H^{1-\beta}_{R,l}(0,1)}
\leq C_1 \big \| I_+^{\beta}v \big \|_{H^1(0,1)}, \quad \forall\ v \in H^{1-\beta}_{R,l}(0,1), \\[0.1in]
C_0 \big \| I_-^{\beta} v \big \|_{H^1(0,1)} \leq \|v\|_{H^{1-\beta}_{R,r}(0,1)}
\leq C_1 \big \| I_-^{\beta} v \big \|_{H^1(0,1)}, \quad \forall\ v \in H^{1-\beta}_{R,r}(0,1).
\end{array}$$
\end{theorem}

\begin{proof} By symmetry, we need only to prove the results for $I^{\beta}_+$. The proof is divided into three parts. First we prove that $I^{\beta}_+$ maps $H^{1-\beta}_{R,l}(0,1)$ continuously into $H^1(0,1)$, i.e., the left half of the first estimate in the theorem holds. Since for any $v \in H^{1-\beta}_{R,l}(0,1)$, 
$$\big \|D {}_0I_x^{\beta} v \big \|_{L^2(0,1)} = | v |_{H^{1-\beta}_{R,l}(0,1)} \leq \|v\|_{H^{1-\beta}_{R,l}(0,1)},$$
we need only to prove the estimate
\begin{equation}\label{iso:e2}
\big \| {}_0I_x^{\beta} v \big \|_{L^2(0,1)} \leq C \| v \|_{H^{1-\beta}_{R,l}(0,1)}, \quad \forall v \in H^{1-\beta}_{R,l}(0,1).
\end{equation}
We use a duality argument for the proof. For any $\phi \in C^\infty_0(0,1)$, we have
$$\begin{array}{l}
\Big |  \big( {}_0I^{\beta}_x v, \phi \big)_{L^2(0,1)} \Big | \\
\ds \quad = \left | \Big ( {}_0I^{\beta}_x v, D\Big (\int^x_0 \phi(s) ds - x \int^1_0\phi ds \Big) \Big)_{L^2(0,1)} \right. \\[0.05in]
\ds \qquad \quad \left. + \Bigl ( {}_0I^{\beta}_x v,\int^1_0\phi(s) ds\Big)_{L^2(0,1)} \right | \\[0.1in]
\ds \quad = \left | - \Big ( D {}_0I^{\beta}_x v, \int^x_0 \phi(s) ds - x \int^1_0 \phi(s) ds \Big)_{L^2(0,1)} \right. \\
\ds \left. \qquad \quad   + \int^1_0\phi(s) ds \Bigl (v, {}_xI^{\beta}_1 1 \Big )_{L^2(0,1)} \right | \\
\ds \quad \le \big \| {}^R_0D^{1-\beta}_x v \big \|_{L^2(0,1)} \left ( \Big \| \int^x_0\phi(s) ds\|_{L^2(0,1)}
     + \Big |\int^1_0 \phi(x) dx \Big| \right) \\
\ds \qquad + \f1{\Gamma(1+\beta)} \Big |\int^1_0\phi(x) dx \Big | ~\|v\|_{L^\kappa(0,1)} ~\big \|(1-x)^\beta \big \|_{L^{\kappa'}(0,1)}\\[0.15in]
\quad \leq C\|v\|_{H^{1-\beta}_{R,l}(0,1)}\|\phi\|_{L^2(0,1)}.
\end{array}$$
Here we have used \eqref{FDE:e0} with $\alpha = \beta$ and $\mu = 0$. We then divide the inequality by $\| \phi\|_{L^2(0,1)}$ and take the supremum over $\phi \in C^\infty_0(0,1)$ and use the fact that $C^\infty_0(0,1)$ is dense in $L^2(0,1)$ to arrive at \eqref{iso:e2}. We thus have proved the left half of the first estimate in the theorem.

Next we prove the right half of the first estimate in the theorem by proving the estimate
\begin{equation}\label{iso:e3}
\| v \|_{L^\kappa(0,1)} \leq C \big \| {}_0I_x^{\beta} v \big \|_{H^1(0,1)}, \quad \forall\ v \in H^{1-\beta}_{R,l}(0,1).
\end{equation}
For any $\phi \in C^\infty_0(0,1)$, we use \eqref{Aux2:e2} to obtain
\begin{equation}\label{iso:e4}\begin{array}{l}
\big |(v,\phi)_{L^2} \big | \\
\ds \quad = \Big | \big (v, {}_xI^{\beta}_1 ~{}^C_xD^\beta_1 \phi \big)_{L^2} \Big |
= \Big | \big ( {}_0I^{\beta}_x v, {}^{R}_xD^\beta_1 \phi \big )_{L^2} \Big | \\
\ds \quad = \Big | - \Big ( {}_0I^{\beta}_x v, D \big ( {}_xI^{1-\beta}_1 \phi(x) - (1-x) {}_0I^{1-\beta}_1 \phi \big) 
- {}_0I^{1-\beta}_1 \phi \Big)_{L^2} \Big |\\
\ds \quad = \Big | \Big ( D {}_0I^{\beta}_x v, {}_xI^{1-\beta}_1 \phi - (1-x) {}_0I^{1-\beta}_1 \phi \Big )_{L^2} \\
\ds \qquad \qquad - \big ( {}_0I^{\beta}_x v, {}_0I^{1-\beta}_1 \phi \big )_{L^2} \Big | \\
\ds \quad \le C \Big( \Big| \int^1_0 {}_0I^{\beta}_x v(x) dx \Big| 
+ \big \| D {}_0I^{\beta}_x v \big \|_{L^2(0,1)}\Big) \big ( \big \| {}_xI^{1-\beta}_1 \phi \big \|_{L^2(0,1)} \\
\ds \qquad \qquad + \big | {}_0I^{1-\beta}_1 \phi \big| \big) \\
\ds \quad \le C \Big( \Big| \int^1_0 {}_0I^{\beta}_x v(x) dx\Big| + 
\big\| D {}_0I^{\beta}_x v \big \|_{L^2(0,1)}\Big) \| \phi \|_{L^{\kappa'}(0,1)}.
\end{array}\end{equation}
Here we have the fact that since $\beta\kappa <1$
\begin{equation}\label{iso:e5}\begin{array}{l}
\big \| {}_xI^{1-\beta}_1 \phi \big \|_{L^2(0,1)} \\[0.1in]
\ds \quad \le \big \| {}_x I^{1-\beta}_1 \phi \big \|_{L^\infty(0,1)} \\[0.05in]
\ds \quad \leq \f1{\Gamma(1-\beta)} \sup_{x \in (0,1)} \int^1_x (s-x)^{-\beta} \big | \phi(s) \big | ds \\[0.1in]
\ds \quad \leq \f1{\Gamma(1-\beta)} \sup_{x \in (0,1)} \left ( \int^1_x (s-x)^{-\beta\kappa}ds \right )^\frac{1}{\kappa} \|\phi\|_{L^{\kappa'}(0,1)} \\[0.1in]
\ds \quad \le C \|\phi\|_{L^{\kappa'}(0,1)}.
\end{array} \end{equation}
We divide \eqref{iso:e4} by $ \|\phi\|_{L^{\kappa'}(0,1)}$ and take supremum for all $\phi \in C_0^\infty(0,1)$ to get \eqref{iso:e3}. We thus have proved the right half of the first estimate in the theorem. \eqref{iso:e3} shows that ${}_0I^{\beta}_x$ is an injection from $H^{1-\beta}_{R,l}(0,1)$ into $H^1(0,1)$.

Finally we prove that ${}_0I^{\beta}_x$ maps $H^{1-\beta}_{R,l}(0,1)$ onto $H^1(0,1)$. For $w \in H^1(0,1)$, let
\begin{equation}\label{iso:e7}
v :={}^C_0 D_x^{\beta} \big ( w(x) - w(0) \big ) = {}_0I_x^{1-\beta} D \big (w(x) - w(0) \big ) = {}_0I_x^{1-\beta} Dw.
\end{equation}
Since $Dw \in L^2(0,1)$, Lemma \ref{lem:Aux0} shows that $v \in L^2(0,1) \hookrightarrow  L^\kappa(0,1)$. We apply \eqref{Aux2:e2} to \eqref{iso:e7} to get
\begin{equation}\label{iso:e8}\begin{aligned}
{}_0I_x^{\beta} v(x) & = {}_0I_x^{\beta} {}_0I_x^{1-\beta}D \big (w(x) - w(0) \big) = {}_0I_x D \big (w(x) - w(0) \big) \\
& = w - w(0) = w - {}_0I_x^{\beta} \Bigl ( \frac{w(0)}{\Gamma(1-\beta)} x^{-\beta} \Bigr). \end{aligned} 
\end{equation}
That is, ${}_0I_x^{\beta} v \in H^1(0,1)$. Together with \eqref{iso:e3} we find $v \in H^{1-\beta}_{R,l}(0,1)$. Since $x^{-\beta} \in H^{1-\beta}_{R,l}(0,1)$, we obtain from \eqref{iso:e8} that
\begin{equation}\label{iso:e9}
v^* := v+\frac{w(0)}{\Gamma(1-\beta)} x^{-\beta} \in H^{1-\beta}_{R,l}(0,1)
\end{equation}
satisfies that ${}_0I^{\beta}_x v^* = w$. That is,  ${}_0I_x^{\beta}$ maps $H^{1-\beta}_{R,l}(0,1)$ onto $H^1(0,1)$. 
\end{proof}

\begin{corollary}\label{cor:Space_Char} For $\ 0 < \beta < 1$ the left and right Riemann--Liouville fractional derivative spaces $H^{1-\beta}_{R,l}(0,1)$ and $H^{1-\beta}_{R,r}(0,1)$ are characterized by
\[\begin{array}{l}
\ds H^{1-\beta}_{R,l}(0,1) = \left \{ {}^{R}_0 D_x^{\beta} w(x) - \frac{w(0)}{\Gamma(1-\beta)} x^{-\beta}: \ \ w \in H^1(0,1)
\right \},\\[0.15in]
\ds H^{1-\beta}_{R,r}(0,1) = \left \{{}^{R}_x D_1^{\beta} w(x) - \frac{w(1)}{\Gamma(1-\beta)} (1-x)^{-\beta}: \ \ w \in H^1(0,1) \right \}.
\end{array}\]
\end{corollary}

\begin{proof} By symmetry we only prove the first equality. From Theorem \ref{thm:iso} ${}_0I_x^{\beta}$ is an isomorphism between $H^{1-\beta}_{R,l}(0,1)$ and $H^1(0,1)$. For any $v \in H^{1-\beta}_{R,l}(0,1)$, there exists a unique $w \in H^1(0,1)$ such that \eqref{iso:e8} holds. We apply ${}^{R}_0D_x^{\beta}$ on \eqref{iso:e8} and use \eqref{Aux2:e2} to finish the proof.
\end{proof}

\begin{corollary}\label{cor:RFred}{\rm (Riemann--Liouville fractional Friedrichs inequality)} Let $0 < \beta <1$. The following estimates hold
$$\begin{aligned}
\|v\|_{L^\kappa(0,1)} & \leq C \Bigl ( \Bigl | \int^1_0 {}_0I_x^{\beta} v(x) dx \Bigr | + \big \|{}^{R}_0 D_x^{1-\beta} v \big \|_{L^2(0,1)} \Bigr ), 
&\forall\ v \in H^{1-\beta}_{R,l}(0,1),\\
\|v\|_{L^\kappa(0,1)} & \leq C \Bigl ( \Bigl | \int^1_0 {}_xI_1^{\beta} v(x) dx \Bigr | + \big \|{}^{R}_x D_1^{1-\beta}v \big \|_{L^2(0,1)} \Bigr ), 
&\forall\ v \in H^{1-\beta}_{R,r}(0,1).
\end{aligned}$$
Consequently, $\|v\|_{H^{1-\beta,0}_{R,l}(0,1)}$ and $\|v\|_{H^{1-\beta,0}_{R,r}(0,1)}$ define norms on $H^{1-\beta,0}_{R,l}(0,1)$ and $H^{1-\beta,0}_{R,r}(0,1)$, respectively.
\end{corollary}

\begin{proof} By symmetry we only prove the first inequality. We divide \eqref{iso:e4} by $\| \phi \|_{L^{\kappa'}(0,1)}$ and take the supremum over $\phi \in C_0^\infty(0,1)$ to finish the proof.
\end{proof}

\begin{lemma}\label{lem:CFred}{\rm (Caputo fractional Friedrichs inequality)} For $0 < \beta <1$,
\[\begin{aligned}\label{CFred:e1}
\|v\|_{L^\kappa(0,1)} & \leq C \Big ( \Big |\int^1_0 v(x) dx \Big | + \big \|{}^{C}_0 D^{1-\beta}_x v \big \|_{L^2(0,1)}\Big ), & \forall\ v \in  H^{1-\beta}_{C,l}(0,1), \\
\|v\|_{L^\kappa(0,1)} & \leq C \Big ( \Big |\int^1_0 v(x) dx \Big | + \big \|{}^{C}_x D^{1-\beta}_1 v \big \|_{L^2(0,1)}\Big ), & \forall\ v \in  H^{1-\beta}_{C,r}(0,1).
\end{aligned}
\]
Consequently, $\|v\|_{H^{1-\beta,0}_{C,l}(0,1)}$ and $\|v\|_{H^{1-\beta,0}_{C,r}(0,1)}$ define norms on $H^{1-\beta,0}_{C,l}(0,1)$ and $H^{1-\beta,0}_{C,r}(0,1)$ respectively.
\end{lemma}

\begin{proof} By symmetry we only prove the first inequality. For any $\phi \in C^\infty_0(0,1)$, we decompose $\phi$ as follows
\[(v,\phi)_{L^2} = - \Big (v, D \Big (\int^1_x \phi(s) ds - (1-x) \int^1_0 \phi(s) ds \Big) \Big)_{L^2} + \Big (v,\int^{1}_0 \phi(s) ds \Big)_{L^2}.
\]
The first term on the right-hand side of the equation can be rewritten as
$$\begin{array}{l}
\ds \Big (Dv, {}_xI^{\beta}_1 ~{}^{C}_xD^{\beta}_1 \Big (\int^{1}_x \phi(s) ds-(1-x)\int^{1}_0 \phi(s) ds \Big ) \Big)_{L^2}\\
\ds \quad = \Big ( {}^{C}_0 D_x^{1-\beta} v, {}_xI_1^{1-\beta} D \Bigl(\int^{1}_x \phi ds - (1-x)\int^{1}_0 \phi(s) ds \Big) \Big)_{L^2}\\
\ds \quad =- \Big ({}^{C}_0 D_x^{1-\beta} v, {}_xI_1^{1-\beta} \Big( \phi -\int^{1}_0 \phi ds \Big) \Big)_{L^2}.
\end{array}$$
We combine the preceding two equations and use \eqref{iso:e4} to get
$$\begin{array}{l}
\big |(v,\phi)_{L^2(0,1)} \big | \\
\quad \ds \leq C \Big ( \Big |\int^1_0vdx \Big |+ \big \| {}^{C}_0 D^{1-\beta}_x v \big \|_{L^2(0,1)}\Big ) \\
\ds \qquad \times \Big ( \big \| {}_xI^{1-\beta}_1 \phi \big \|_{L^2(0,1)} + \Big |\int^1_0 \phi dx \Big | \Big )\\
\ds \quad \le C \Big ( \Big |\int^1_0 v(x) dx \Big |+ \big \|{}^{C}_0 D^{1-\beta}_x v \big \|_{L^2(0,1)}\Big )\|\phi\|_{L^{\kappa'}(0,1)}.
\end{array}$$
We divide this inequality by $\|\phi\|_{L^{\kappa'}(0,1)}$ and take the supremum over $\phi \in C_0^\infty(0,1)$ to finish the proof.
\end{proof}

\begin{theorem}\label{thm:embed} {\rm (Embedding theorem)} For $0 < \beta <1/2$ $H^{1-\beta}_{C,l}(0,1)$ and $H^{1-\beta}_{C,r}(0,1)$ are continuously embedded into the H\"{o}lder space $C^{\f1{2}-\beta}[0,1]$
\begin{equation}\label{embed:e0}\begin{aligned}
  \| v \|_{C^{\f1{2}-\beta}[0,1]} \leq C\|v\|_{H^{1-\beta}_{C,l}(0,1)}, \qquad & \forall v \in H^{1-\beta}_{C,l}(0,1), \\
  \| v \|_{C^{\f1{2}-\beta}[0,1]} \leq C\|v\|_{H^{1-\beta}_{C,r}(0,1)}, \qquad & \forall v \in H^{1-\beta}_{C,r}(0,1).
\end{aligned}\end{equation}
\end{theorem}

\begin{proof} First we prove that for each $ v \in C^\infty(0,1) \bigcap H^{1-\beta}_{C,l}(0,1)$, $v$ is H\"{o}lder continuous in $(0,1)$ with index $1/2-\beta$ so $v$ can be continuously extended to $[0,1]$. For any $x, y \in (0,1)$, without loss of generality, we assume that $ 0 < x <y <1$. We use \eqref{FDS:e5} and \eqref{Aux1:e1} (on $[0,x]$ and $[0,y]$) to obtain
\begin{equation}\label{embed:e1}\begin{array}{l}
\ds v(y) - v(x) = \int_x^y Dv(s) ds = \int_0^y Dv(s) ds - \int_0^x Dv(s) ds \\[0.1in]
\ds \quad = \frac{1}{\Gamma(1-\beta)} \Big[ \int^y_0 {}_sI^{\beta}_y(y-s)^{-\beta}Dvds - \int^x_0 {}_sI^{\beta}_x(x-s)^{-\beta}Dvds  \Big] \\[0.15in]
\ds \quad = \frac{1}{\Gamma(1-\beta)}\Big[ \int^y_0 (y-s)^{-\beta}{}^{}_0I^{\beta}_sDvds
      - \int^x_0 (x-s)^{-\beta}{}_0I^{\beta}_sDvds \Big] \\[0.15in]
\ds \quad = \frac{1}{\Gamma(1-\beta)}\Big[\int_x^y (y-s)^{-\beta}{}^{C}_0D^{1-\beta}_s vds \\[0.1in]
\ds \qquad + \int^x_0 \big((y-s)^{-\beta}-(x-s)^{-\beta}\big){}^{C}_0D^{1-\beta}_s vds \Big].
\end{array}\end{equation}

\begin{equation}\label{embed:e2}
\begin{aligned}
\Bigl | \int_x^y (y-s)^{-\beta}{}^{C}_0D^{1-\beta}_s vds \Big | & \le \Bigl ( \int_x^y (y-s)^{-2\beta} ds \Big )^{1/2} \big \|{}^{C}_0D^{1-\beta}_s v \big \|_{L^2(0,1)} \\
& \le C | y - x |^{\f1{2} - \beta} \big \|{}^{C}_0D^{1-\beta}_s v \big \|_{L^2(0,1)}. \end{aligned} \end{equation}
We let $\theta = (x-s)/(y-x)$ to obtain
\begin{equation}\label{embed:e3}\begin{aligned}
& \int^x_0 \big ((y-s)^{-\beta}-(x-s)^{-\beta} \big )^2ds \\
& \quad = (y - x)^{1-2\beta} \int^x_0\Big(\big(1+\frac{x-s}{y-x}\big)^{-\beta}-\big(\frac{x-s}{y-x}\big)^{-\beta}\Big)^2 d\frac{x-s}{y-x}\\
& \quad = (y - x)^{1-2\beta} \int^{x/(y-x)}_0 ((1+\theta)^{-\beta} - \theta^{-\beta})^2 d\theta \\
& \quad \leq C (y - x)^{1-2\beta} \int_0^\infty \left (\f1{\theta^{\beta}} - \f1{(1+\theta)^{\beta}} \right)^2 d\theta 
\le C (y - x)^{1-2\beta}.
\end{aligned}\end{equation}
It is clear that the improper integral on the right-hand side converges near $\theta = 0$. Tts convergence at $\theta = \infty$ is ensured by the following expansion for $\theta >>1$
\begin{equation*}
\f1{\theta^{\beta}}-\f1{(1+\theta)^\beta}
= \ds \f{(1+\theta)^\beta - \theta^{\beta}}{\theta^{\beta}(1+\theta)^\beta}
= \f{\ds \left(1+\f1{\theta}\right)^\beta - 1}{\ds \theta^{\beta}\left(1+\f1{\theta} \right)^\beta}
= \ds \f{\ds \beta + O\left(\f1{\theta}\right)}{\ds \theta^{1+\beta}\left(1+\f1{\theta}\right)^\beta}.
\end{equation*}
We combine the preceding estimates to obtain
\begin{equation}\label{embed:e4}
|v(x)-v(y)| \leq C|x-y|^{\f1{2}-\beta} \big \|{}^{C}_0D_x^{1-\beta}v \big \|_{L^2(0,1)}, \quad \ \forall x,y \in (0,1).
\end{equation}
Thus, $v$ is uniformly continuous in $(0,1)$ so $v$ can be continuously extended to $[0,1]$.  

Next for each $v \in C[0,1] \bigcap H^{1-\beta}_{C,l}(0,1) \bigcap C^\infty(0,1)$, we use \eqref{FDE:e0} with $\alpha = \beta$ and $\mu = 1 - \beta$ and \eqref{Aux1:e1} to get
\[\begin{aligned}
| v(0) | & = \Bigl | - \int^1_0 D \big ((1-x)v(x) \big)dx=-\int^1_0 \big ((1-x)Dv-v \big)dx \Bigr | \\
& = \Bigl | - \frac{1}{\Gamma(2-\beta)} \int^1_0 {}^{}_xI^{\beta}_1(1-x)^{1-\beta}Dvdx +\int^1_0 vdx \Bigr | \\
& = \Bigl | - \frac{1}{\Gamma(2-\beta)}\int^1_0 (1-x)^{1-\beta}{}^{}_0 I^{\beta}_xDvdx + \int^1_0 vdx \Bigr |  \leq C\|v\|_{H^{1-\beta}_{C,l}(0,1)}.
\end{aligned}\]
We combine this estimate with \eqref{embed:e4} to prove \eqref{embed:e0} holds for any $v \in  H^{1-\beta}_{C,l}(0,1) \cap C^\infty(0,1) \cap C[0,1]$.

Finally, for any $v \in H^{1-\beta}_{C,l}(0,1)$, since $C^\infty(0,1) \cap H^{1-\beta}_{C,l}(0,1)$ is dense in $H^{1-\beta}_{C,l}(0,1)$ and is equal to $C^\infty(0,1) \cap H^{1-\beta}_{C,l}(0,1) \cap C[0,1]$, there exists a sequence of functions $\{v_n\}^\infty_{n=1} \subset C^\infty(0,1) \cap H^{1-\beta}_{C,l}(0,1) \cap C[0,1]$, which converges to $v$ in $H^{1-\beta}_{C,l}(0,1)$. It follows from \eqref{embed:e0} that $\{v_n\}^\infty_{n=1}$ is a Cauchy sequence in $C^{\f1{2}-\beta}[0,1]$. The completeness of $C^{\f1{2}-\beta}[0,1]$ concludes that $v \in C^{\f1{2}-\beta}[0,1]$ and \eqref{embed:e0} holds.
\end{proof}

\begin{remark} This theorem shows that the Caputo fractional derivative spaces $H^{1-\beta}_{C,l}(0,1)$ and $H^{1-\beta}_{C,r}(0,1)$ have similar embedding properties to the standard fractional Sobolev spaces $H^{1-\beta}(0,1)$. However,  such embedding does not hold for the Riemann--Liouville fractional derivative spaces $H^{1-\beta}_{R,l}(0,1)$ and $H^{1-\beta}_{R,r}(0,1)$. For example, $x^{-\beta} \in H^{1-\beta}_{R,l}(0,1)$ for $0 < \beta < 1$ (see \eqref{FDS:e5}) but $x^{-\beta} \notin C[0,1]$. 
\end{remark}

Finally we apply Theorem \ref{thm:embed} to extend Lemma \ref{lem:Aux2} to $H^{1-\beta}_{C,l}(0,1)$ and $H^{1-\beta}_{C,r}(0,1)$.

\begin{lemma}\label{lem:Aux3}
Let $0 < \beta < 1/2$. For any $v \in H^{1-\beta}_{C,l}(0,1)$ with $v(0)=0$ and $w \in H^{1-\beta}_{C,r}(0,1)$ with $w(1)=0$, identities \eqref{Aux2:e2} hold. 
\end{lemma}
\begin{proof} By symmetry, we only prove the estimates for $v$. For any $v \in H^{1-\beta}_{C,l}(0,1)$ with $v(0)=0$ and any $\phi \in C^\infty_0(0,1)$, there exists a sequence of functions $\{v_n\}_{n=1}^\infty \subset C^\infty(0,1) \cap H^{1-\beta}_{C,l}(0,1) \cap C[0,1]$ with $v_n(0)=0$ for any $ 1 \leq n < \infty$ such that $v_n$ converges to $v$ in $H^{1-\beta}_{C,l}(0,1)$. We use \eqref{Aux2:e2} for $\phi$, \eqref{Aux1:e1}, and the fact that   $v_n(0)=0$ and ${}_xI^{1-\beta}_1 \phi=0$ at $x=1$ to obtain
$$\begin{array}{rl}
\ds \big ({}^{C}_0D^{\beta}_x v, \phi \big)_{L^2(0,1)} & \ds = \lim_{n \rightarrow \infty} \big ({}^{C}_0D^{\beta}_x v_n,\phi \big )_{L^2(0,1)} = \lim_{n \rightarrow \infty} \big (D v_n,{}^{}_xI^{1-\beta}_1\phi \big)_{L^2(0,1)} \\
&\ds = - \lim_{n \rightarrow \infty} \big( v_n, {}^{R}_xD^{1-\beta}_1 \phi \big)_{L^2(0,1)} 
= - \big ( v, {}^{R}_xD^{1-\beta}_1 \phi \big)_{L^2(0,1)} \\
&\ds = - \big ( v, {}^{C}_xD^{1-\beta}_1 \phi \big)_{L^2(0,1)} 
   = - \big ( {}^{}_0I^{1-\beta}_x v,D\phi \big )_{L^2(0,1)} \\
& \ds = \big ( {}^{R}_0D^{\beta}_x v, \phi \big )_{L^2(0,1)}.
\end{array}$$
Therefore, ${}^{C}_0D^{\beta}_x v={}^{R}_0D^{\beta}_xv$ and \eqref{Aux2:e2} holds. Similarly, we have 
\[ \begin{aligned}
& \big ({}_0I^{\beta}_x{}^{C}_0D^{\beta}_x v, \phi \big )_{L^2(0,1)} =
\big ({}^{C}_0D^{\beta}_x v,{}_xI^{\beta}_1 \phi \big)_{L^2(0,1)} \\
& \quad = \lim_{n \rightarrow \infty} \big ({}^{C}_0D^{\beta}_x v_n,{}^{}_xD^{-\beta}_1\phi \big)_{L^2(0,1)} 
=\lim_{n \rightarrow \infty} \big ( D v_n,{}_xI^{1-\beta}_1{}^{}_xI^{\beta}_1 \phi \big )_{L^2(0,1)} \\
& \quad = \lim_{n \rightarrow \infty} \big (D v_n, {}_xI_1 \phi \big )_{L^2(0,1)} = - \lim_{n \rightarrow \infty} \big (v_n, D{}_xI_1\phi \big)_{L^2(0,1)} = ( v,\phi)_{L^2(0,1)}.
\end{aligned} \]
Therefore, ${}_0I^{\beta}_x{}^{R}_0D^{\beta}_x v={}_0I^{\beta}_x{}^{C}_0D^{\beta}_x v=v$ for any $ v \in H^{1-\beta}_{C,l}(0,1)$. Similarly, we can prove the rest equality in \eqref{Aux2:e2}.  
\end{proof}

We are now in the position to study the wellposedness of the different combinations of FDEs \eqref{FDE:Caputo}, \eqref{FDE:Conserv}, and \eqref{FDE:RL} enclosed with the different Neumann boundary conditions \eqref{FDE:Nbc}, \eqref{FDE:Cbc}, and \eqref{FDE:Rbc}.

\section{The Caputo FDE}
\setcounter{section}{4}\setcounter{equation}{0}

We study the Caputo FDE \eqref{FDE:Caputo} enclosed with the classical Neumann boundary condition \eqref{FDE:Nbc}. We multiply \eqref{FDE:Caputo} by any $v \in H^{1-\beta}_{R,r}(0,1)$, use Lemma \ref{lem:Aux1}, integrate the resulting equation by parts from 0 to 1, and incorporate the boundary condition \eqref{FDE:Nbc} to obtain a Petrov-Galerkin weak formulation: find $u \in H^1(0,1)$ such that
\begin{equation}\label{Caputo:e1}
A_C(u,v) = l_C(v), \quad \forall\ v \in H^{1-\beta}_{R,r}(0,1),
\end{equation}
where the bilinear form $A_C(\cdot,\cdot): H^1(0,1) \times H^{1-\beta}_{R,r}(0,1) \longrightarrow \mathbb{R}$ and the right-hand side $l_C(\cdot): H^{1-\beta}_{R,r}(0,1) \longrightarrow \mathbb{R}$ are defined to be
\begin{equation}\label{Caputo:e1a}\begin{array}{rl}
A_C(w,v) & := - \bigl (Dw, {}^{R}_xD_1^{1-\beta}v \big )_{L^2(0,1)}, \\[0.1in]
l_C(v) & := \langle f, v \rangle + a_1~{}_xI_1^{\beta}v|_{x=1} - a_0~{}_xI_1^{\beta}v|_{x=0}.
\end{array}\end{equation}

\begin{theorem}\label{thm:Caputo} Let $0 < \beta < 1$. Assume $f \in (H^{1-\beta}_{R,r}(0,1))'$ satisfies the constraint
\begin{equation}\label{Caputo:e2}
\langle f, (1-x)^{-\beta}\rangle + \Gamma(1-\beta)  ( a_1 - a_0 ) = 0.
\end{equation}
Then \eqref{Caputo:e1} has a unique solution $u^* \in H^{1,0}(0,1)$ with a stability estimate
\begin{equation}\label{Caputo:e3}
\| u^* \|_{H^1(0,1)} \le C \big ( \|f\|_{(H^{1-\beta,0}_{R,r}(0,1))^\prime}+|a_0|+|a_1| \big ).
\end{equation}
Furthermore, any weak solution $u \in H^1(0,1)$ to \eqref{Caputo:e1} can be expressed as $u = u^* + C$ with $C \in \mathbb{R}$ being an arbitrary constant, and vice versa.
\end{theorem}

\begin{proof} Since $(1-x)^{-\beta} \in L^\kappa(0,1)$ and ${}^{}_xI_1^{\beta}(1-x)^{-\beta} = \Gamma(1-\beta)$, $(1-x)^{-\beta} \in H^{1-\beta}_{R,r}(0,1)$. Taking $v=(1-x)^{-\beta}$ in \eqref{Caputo:e1} concludes that \eqref{Caputo:e2} is a necessary condition for the existence of a weak solution.

To prove the existence of a weak solution to \eqref{Caputo:e1}, we introduce an auxiliary Galerkin formulation: find $\xi \in H^{1-\beta,0}_{R,r}(0,1)$ such that
\begin{equation}\label{Caputo:e4}
A_{R,r}(\xi,v) = l_C(v), \quad \forall v \in H^{1-\beta,0}_{R,r}(0,1),
\end{equation}
where $A_{R,r}(\cdot,\cdot)$ is defined by
\begin{equation*}
A_{R,r}(\xi,v) := \big ({}^{R}_xD_1^{1-\beta}\xi, {}^{R}_xD_1^{1-\beta}v \big)_{L^2(0,1)}.
\end{equation*}
It is clear that for any $w$ and $v$ in $H^{1-\beta,0}_{R,r}(0,1)$
\begin{equation*}
\big | A_{R,r}(\xi,v) \big | \leq \|{}^{R}_xD_1^{1-\beta}\xi\|_{L^2(0,1)}\|{}^{R}_xD_1^{1-\beta}v \|_{L^2(0,1)} =\|\xi\|_{H^{1-\beta,0}_{R,r}(0,1)}\|v \|_{H^{1-\beta,0}_{R,r}(0,1)}.
\end{equation*}
Moreover, Corollary \ref{cor:RFred} concludes that 
\begin{equation*}
A_{R,r}(\xi,\xi) = \|{}^{R}_xD_1^{1-\beta}\xi\|^2_{L^2(0,1)} = \|\xi\|^2_{H^{1-\beta,0}_{R,r}(0,1)}
\end{equation*}
is coercive on $H^{1-\beta,0}_{R,r}(0,1) \times H^{1-\beta,0}_{R,r}(0,1)$. Furthermore, Theorem \ref{thm:iso} states that for any $v \in H^{1-\beta,0}_{R,r}(0,1)$, ${}_xI_1^{\beta}v \in H^1(0,1) \hookrightarrow C[0,1]$. Thus, ${}_xI_1^{\beta}v|_{x=0}$ and ${}_xI_1^{\beta}v|_{x=1}$ can be bounded by $\|{}_xI^{\beta}_1 v\|_{H^1(0,1)}$ and so $\| v \|_{H^{1-\beta,0}_{R,r}(0,1)}$. Hence,
\[ \begin{aligned}
\big | l_C(v) \big | & = \big | \langle f, v \rangle + a_1~{}_xI_1^{\beta}v|_{x=1} - a_0~{}_xI_1^{\beta}v|_{x=0} \big | \\
& \leq C \big ( \|f\|_{(H^{1-\beta}_{R,r}(0,1))^\prime}+|a_0|+|a_1| \big ) \| v \|_{H^{1-\beta,0}_{R,r}}.
\end{aligned} \]
Lax--Milgram theorem shows that the auxiliary problem \eqref{Caputo:e3} has a unique solution $\xi^* \in H^{1-\beta,0}_{R,r}(0,1)$ with a stability estimate
\begin{equation*}
\|D{}^{}_xI^{\beta}_1 \xi^* \|_{L^2(0,1)}=\| \xi^* \|_{H^{1-\beta,0}_{R,r}(0,1)} \le C  \big ( \|f\|_{(H^{1-\beta,0}_{R,r}(0,1))^\prime}+|a_0|+|a_1| \big ).
\end{equation*}
Then $u^* := {}_xI_1^{\beta} \xi^* \in H^{1,0}(0,1)$ satisfies \eqref{Caputo:e1} for any $v \in H^{1-\beta,0}_{R,r}(0,1)$ and \eqref{Caputo:e3}. For any $v \in H^{1-\beta}_{R,r}(0,1)$ we use the condition \eqref{Caputo:e2} to conclude that
\begin{equation}\label{Caputo:e5}
v^0 := v - \f{\ds \int_0^1 {}_xI_1^{\beta}v(x) dx}{\Gamma(1-\beta)} (1-x)^{-\beta} \in H^{1-\beta,0}_{R,r}(0,1)
\end{equation}
satisfies
\begin{equation*}
A_C(u^*,v) = A_C(u^*,v^0), \quad l_C(v) = l_C(v^0).
\end{equation*}
Thus, $u^*$ satisfies \eqref{Caputo:e1}. Clearly, any $u = u^* + C$ is also a weak solution to \eqref{Caputo:e1}.

Conversely, for any weak solution $u \in H^1(0,1)$ to problem \eqref{Caputo:e1}, we deduce from Theorem \ref{thm:iso} that there exists a unique $\xi \in H^{1-\beta}_{R,r}(0,1)$ with ${}_xI_1^{\beta} \xi = u$ satisfies \eqref{Caputo:e4} for any $v \in H^{1-\beta}_{R,r}(0,1)$. If we relate $\xi^0$ and $\xi$ in the same way as $v^0$ and $v$ in \eqref{Caputo:e5}, we find that $\xi^0 \in H^{1-\beta,0}_{R,r}(0,1)$ satisfies \eqref{Caputo:e4}. The uniquess of the weak solution in $H^{1-\beta,0}_{R,r}(0,1)$ to \eqref{Caputo:e4} ensures that $\xi^0 = \xi^*$ with $u^* = {}_xI_1^{\beta} \xi^* \in H^{1,0}(0,1)$ as introduced above. Consequently, we use Lemma \ref{lem:Aux3} to obtain
\[ \begin{aligned}
u = {}_xI_1^{\beta} \xi = {}_xI_1^{\beta} \xi^* + \int_0^1 {}_xI_1^{\beta}\xi(x) dx = u^* + \int_0^1 u dx = u^* + C. 
\end{aligned} \]
We thus finish the proof of the theorem. 
\end{proof}

\begin{remark} 
We will show in \S \ref{CounterCaputo} that the Caputo equation \eqref{FDE:Caputo} with the Caputo boundary condition \eqref{FDE:Cbc} or the Riemann--Liouville boundary condition \eqref{FDE:Rbc} does not admit a solution, in general.
\end{remark}

\section{The conservative Caputo FDE}
\setcounter{section}{5}\setcounter{equation}{0}

We consider the conservative Caputo FDE \eqref{FDE:Conserv} with the Caputo boundary condtion \eqref{FDE:Cbc}. We multiply \eqref{FDE:Conserv} by any $v \in H^1(0,1)$, integrate the equation from 0 to 1, incorporate the boundary condition \eqref{FDE:Cbc}, and use Lemma \ref{lem:Aux1} to obtain a Galerkin weak formulation: find $u \in H^{1-\beta/2}_{C,l}(0,1)$ such that
\begin{equation}\label{ConCbc:e1}
B(u,v) = l(v), \quad \forall v \in H^{1-\beta/2}_{C,l}(0,1)
\end{equation}
where the bilinear form $B(\cdot,\cdot): H^{1-\beta/2}_{C,l}(0,1) \times H^{1-\beta/2}_{C,l}(0,1) \rightarrow \mathbb{R}$ and the linear functional $l(\cdot): H^{1-\beta/2}_{C,l}(0,1) \rightarrow \mathbb{R}$ are defined to be
\begin{equation*}\begin{array}{rl}
 B(w,v) &:= - \big ( {}^{C}_0D_x^{1-\beta/2}w, {}^{C}_xD_1^{1-\beta/2}v \big)_{L^2},\\[0.1in]
l(v) &:= \langle f, v \rangle + a_1 ~v(1) - a_0 ~v(0).
\end{array}\end{equation*}

\begin{theorem}\label{thm:ConCbc} Let $0 < \beta < 1$ and $f \in (H^{1-\beta/2}_{C,l}(0,1))^\prime$ satisfy the constraint:
\begin{equation}\label{ConCbc:e2}
\langle f, 1 \rangle + a_1 - a_0 = 0.
\end{equation}
Then \eqref{ConCbc:e1} has a unique weak solution $u^* \in H^{1-\beta/2,0}_{C,l}(0,1)$ with a stability estimate
\begin{equation}\label{ConCbc:e3}
\| u^* \|_{H^{1-\beta/2}_{R,l}} \le C \bigl (\|f\|_{(H^{1-\beta/2}_{C,l}(0,1))^\prime} + |a_0| + |a_1| \bigr).
\end{equation}
Any weak solution $u \in H^{1-\beta/2}_{C,l}(0,1)$ to \eqref{ConCbc:e1} is of the form $u = u^* + C$ and vice versa.
\end{theorem}

\begin{proof} Theorem \ref{thm:embed} shows that $H^{1-\beta/2}_{C,l}(0,1) \hookrightarrow C[0,1]$ for $0 < \beta < 1$, so \eqref{ConCbc:e1} is well defined and $l(\cdot)$ is a bounded linear functional on the space $H^{1-\beta/2,0}_{C,l}(0,1)$. Taking $v\equiv 1 \in H^{1-\beta/2}_{C,l}(0,1)$ in \eqref{ConCbc:e1} concludes that \eqref{ConCbc:e2} is a necessary condition for the existence of a weak solution to \eqref{ConCbc:e1}. On the other hand, we get from \cite{ErvRoo} that
\[\begin{array}{rl}
B(w,w) & = \ds \big ( {}_0I_x^{\beta/2}Dw, {}_xI_1^{\beta/2} Dw \big)_{L^2(0,1)}
= \ds \cos\Bigl(\f{\beta \pi}2\Bigr) \bigl \|{}_0I_x^{\beta/2} Dw \bigr \|^2_{L^2(0,1)} \\[0.1in]
& = \ds \cos\Bigl(\f{\beta \pi}2\Bigr) \bigl \|{}^{C}_0D_x^{1-\beta/2}w \bigr \|^2_{L^2(0,1)},
\quad \forall\ w \in H^{1-\beta/2}_{C,l}(0,1).
\end{array}\]
It follows from Lemma \ref{lem:CFred} that the bilinear form $B(\cdot,\cdot)$ is coercive and bounded on  $H^{1-\beta/2,0}_{C,l}(0,1) \times H^{1-\beta/2,0}_{C,l}(0,1)$. We apply Lax-Milgram theorem to conclude that there exists a unique solution $u^* \in H^{1-\beta/2,0}_{C,l}(0,1)$ satisfies \eqref{ConCbc:e1} for any $v \in H^{1-\beta/2,0}_{C,l}(0,1)$ and the stability estimate \eqref{ConCbc:e3}. For any $v \in H^{1-\beta/2}_{C,l}(0,1)$, we use the condition \eqref{ConCbc:e2} to deduce that
\[ v^0 := v - \int_0^1 v dx \in H^{1-\beta/2,0}_{C,l}(0,1)\]
satisfies $B(u^*,v) = B(u^*,v^0)$ and $l(v) = l(v^0)$. Hence, $u^* \in H^{1-\beta/2,0}_{C,l}(0,1)$ is a solution to \eqref{ConCbc:e1}. Clearly, any $u = u^* + C \in  H^{1-\beta/2}_{C,l}(0,1)$ is a weak solution to \eqref{ConCbc:e1}.

Conversely, let $u \in  H^{1-\beta/2}_{C,l}(0,1)$ be any weak solution to \eqref{ConCbc:e1}. Then $w := u - u^* \in  H^{1-\beta/2}_{C,l}(0,1)$ satisfies
\[ B(w,w) = \ds \cos\Bigl(\f{\beta \pi}2\Bigr) \bigl \|{}^{C}_0D_x^{1-\beta/2}w \bigr \|^2_{L^2(0,1)} = 0. \]
Hence, ${}^{C}_0D_x^{1-\beta/2}w = {}_0I_x^{\beta/2} Dw= 0$. We apply ${}^{R}_0D_x^{\beta/2} = D {}_0I_x^{1-\beta/2}$ on both sides of the equation to get
\[ Dw = D {}_0I_x Dw = D {}_0I_x^{1-\beta/2} {}_0I_x^{\beta/2} Dw = {}^{R}_0D_x^{\beta/2} {}_0I_x^{\beta/2} D w= 0. \]
Therefore, $w$ is a constant and the theorem is proved.
\end{proof}

The next theorem characterizes the solution to the FDE in terms of the solution to the classical Dirichlet boundary-value problem of a second-order diffusion equation.

\begin{theorem}\label{thm:ConCbc2} Assume that the conditions of Theorem \ref{thm:ConCbc} hold and $f \in L^2(0,1)$. Then, any weak solution $u \in H^{1-\beta/2}_{C,l}(0,1)$ to \eqref{ConCbc:e1} can be expressed as
\begin{equation}\label{ConCbc2:e1}
u(x) = {}_0 I^{1-\beta}_x w(x) + u(0)
\end{equation}
where $w$ is the solution to the Dirichlet boundary-value problem of the second-order diffusion equation
\begin{equation}\label{ConCbc2:e2}
\left\{\begin{aligned}
& - D^2w = Df(x), \quad  x \in (0,1), \\
& w(0) = a_0, \ \  w(1) = a_1.
\end{aligned}\right.
\end{equation}
\end{theorem}

\begin{proof} We note that 
\begin{equation}\label{ConCbc2:e3}
 w := {}^{C}_0D_x^{1-\beta}u = {}^{C}_0D_x^{1-\beta} \big ( u - u(0) \big)
\end{equation}
is a weak solution to \eqref{ConCbc2:e2}. It is well-known that for $f \in L^2(0,1)$ the Dirichlet boundary-value problem \eqref{ConCbc2:e2} has a unique solution in $H^1(0,1)$. We apply ${}_0 I^{1-\beta}_x$ on both sides of \eqref{ConCbc:e1} and use \eqref{Aux2:e2} to get
\[
u-u(0) = {}_0 I^{1-\beta}_x {}^{C}_0D_x^{1-\beta} \big (u - u(0) \big ) = {}_0I_x^{1-\beta}w.
\]
\end{proof}

\section{The Riemann--Liouville FDE}
\setcounter{section}{6}\setcounter{equation}{0}

We consider the Riemann-Liouville FDE \eqref{FDE:RL} with the Riemann--Liouville boundary condition \eqref{FDE:Rbc}. We multiply \eqref{FDE:RL} by ${}_0I_x^{\beta}v$ for any $v \in H^{1-\beta}_{R,l}(0,1)$, integrate the resulting equation from 0 to 1, and incorporate the boundary condition \eqref{FDE:Rbc} to obtain a Galerkin weak formulation: find $u \in H^{1-\beta}_{R,l}(0,1)$ such that
\begin{equation}\label{RR:e1}
A_{R,l}(u,v) = l_{R,l}(v), \quad \forall\ v \in H^{1-\beta}_{R,l}(0,1)
\end{equation}
where the bilinear form $A_{R,l}(\cdot,\cdot): H^{1-\beta}_{R,l}(0,1) \times H^{1-\beta}_{R,l}(0,1) \longrightarrow \mathbb{R}$ and the right-hand side $l_{R,l}(\cdot): H^{1-\beta}_{R,l}(0,1) \longrightarrow \mathbb{R}$ are defined to be
\begin{equation*}\begin{array}{rl}
A_{R,l}(w,v) & := \bigl ({}^{R}_0D_x^{1-\beta}w,{}^{R}_0D_x^{1-\beta}v \big )_{L^2}, \\[0.1in]
l_{R,l}(v) & := \bigl ( f,{}_0I_x^{\beta}v \big )_{L^2} + a_1~{}_0I_x^{\beta}v|_{x=1} - a_0~{}_0I_x^{\beta}v|_{x=0}.
\end{array}\end{equation*}

\begin{theorem}\label{thm:RR} Let $0 < \beta < 1$ and $f \in (H^1(0,1))^\prime$ satisfy the constraint \eqref{ConCbc:e2}. Then the Galerkin formulation \eqref{RR:e1} has a unique solution $u^* \in H^{1-\beta,0}_{R,l}(0,1)$ with a stability estimate
\begin{equation}\label{RR:e2}
\| u^* \|_{H^{1-\beta}_{R,l}} \le C \bigl (\|f\|_{(H^1(0,1))^\prime} + |a_0| + |a_1| \bigr).
\end{equation}
Moreover, any solution $u \in H^{1-\beta}_{R,l}(0,1)$ to problem \eqref{FDE:RL} and \eqref{FDE:Rbc} can be expressed as $u = u^* + Cx^{-\beta}$ for any $C \in \mathbb{R}$ and vice versa. 
\end{theorem}

\begin{proof} By \eqref{FDS:e5}, $x^{-\beta} \in H^{1-\beta}_{R,l}(0,1)$. Taking $v=x^{-\beta}$ in \eqref{RR:e1} concludes that \eqref{ConCbc:e2} is a necessary condition for the existence of a weak solution. We use Lemma \ref{lem:Fred} and Theorem \ref{thm:iso} to bound $\big \|{}_0I^{\beta}_x v \big \|_{H^1(0,1)}$ by $ \big \|{}^{R}_0D^{1-\beta}_x v \big \|_{L^2(0,1)}$ to obtain the estimate
\[
\begin{aligned}
\big | l_{R,l}(v) \big | = & \big | \langle f, {}_0I^{\beta}_x v \rangle + a_1~{}_0I_x^{\beta}v|_{x=1}
    - a_0~{}_0I_x^{\beta}v|_{x=0} \big | \\
& \leq C \big (\|f\|_{(H^1(0,1))'}+|a_0|+|a_1| \big ) \big \|{}_0I^{\beta}_x v \big \|_{H^1(0,1)} \\
& \leq C \big (\|f\|_{(H^1(0,1))'}+|a_0|+|a_1| \big ) \big \|{}^{R}_0D^{1-\beta}_x v \big \|_{L^2(0,1)} \\
& \leq C \big (\|f\|_{(H^1(0,1))'}+|a_0|+|a_1| \big ) \big \| v \big \|_{H^{1-\beta}_{R,l}(0,1)}.
\end{aligned}
\]
Thus, $l_{R,l}(\cdot)$ defines a bounded linear functional on $H^{1-\beta,0}_{R,l}(0,1)$. The coercivity and boundedness of $A_{R,l}(\cdot,\cdot)$ on $H^{1-\beta,0}_{R,l}(0,1) \times H^{1-\beta,0}_{R,l}(0,1)$ follows from $A_{R,l}(v,v) = \| v \|_{H^{1-\beta,0}_{R,l}(0,1)}^2$ and the fact $\| \cdot \|_{H^{1-\beta,0}_{R,l}(0,1)}$ is a norm on $H^{1-\beta,0}_{R,l}(0,1)$.
We apply Lax-Milgram theorem to conclude that there is a unique $u^* \in H^{1-\beta,0}_{R,l}(0,1)$ which satisfies \eqref{RR:e1} for any $v^0 \in H^{1-\beta,0}_{R,l}(0,1)$ and the stability estimate \eqref{RR:e2}. For any $v \in H^{1-\beta}_{R,l}(0,1)$, we use \eqref{ConCbc:e2} to conclude that 
\begin{equation*}
v^0 := v - \f{\ds \int_0^1 {}_0I_x^{\beta}v(x) dx}{\Gamma(1-\beta)} x^{-\beta} \in H^{1-\beta,0}_{R,l}(0,1)
\end{equation*}
satisfies $A_{R,l}(u^*,v) = A_{R,l}(u^*,v^0)$ and $l_{R,l}(v) = l_{R,l}(v^0)$. We thus prove that $u^*$ is a weak solution to \eqref{RR:e1}.

Let $u \in H^{1-\beta}_{R,l}(0,1)$ be any weak solution to \eqref{RR:e1}. Then the difference $w := u - u^*$ satisfies the weak formulation
\begin{equation}\label{RR:e3}
A_{R,l}(w,v) = 0, \quad \forall v \in H^{1-\beta}_{R,l}(0,1).
\end{equation}
Choosing $v = w$ yields $\big \| {}^{R}_0D_x^{1-\beta}w \big \|_{L^2(0,1)} = \big \| D~{}_0I_x^{\beta}w \big \|_{L^2(0,1)} = 0$. That is, ${}_0I^{\beta}_x w = C_0$. We apply ${}^{R}_0D^{\beta}_x = D {}^{}_0I^{1-\beta}_x$ to both sides of the equation to obtain
\begin{equation*}
w(x) = D {}^{}_0I_x w = D {}_0I^{1-\beta}_x {}_0I^{\beta}_x w = D ~{}_0I^{1-\beta}_x C_0 = \f{C_0 x^{-\beta}}{\Gamma(1-\beta)} = C x^{-\beta}.
\end{equation*}
Conversely, $u = u^* + Cx^{-\beta}$ satisfies \eqref{RR:e1}. 
\end{proof}

The next theorem characterizes the solution to the FDE in terms of the solutions to the Neumann boundary-value problem of a second-order diffusion equation.

\begin{theorem}\label{thm:RR_Char} Assume that the conditions of Theorem \ref{thm:RR} hold. Then any solution to problem \eqref{FDE:RL} and \eqref{FDE:Rbc} can be expressed as
\begin{equation}\label{RR_Char:e1}
u = {}^R_0 D^{\beta}_x \bigl( w - w(0) \bigr) + C x^{-\beta},
\end{equation}
where $C \in \mathbb{R}$ is an arbitrary constant and $w$ is any solution of the classical Neumann-boundary value problem of a second-order diffusion equation
\begin{equation}\label{RR_Char:e2}
\left\{\begin{aligned}
& - D^2 w = f(x), \quad  x \in (0,1), \\
& Dw(0) =a_0, \quad Dw(1) = a_1.
\end{aligned}
\right. \end{equation}
Conversely, \eqref{RR_Char:e1} defines the general solution to \eqref{FDE:RL} and \eqref{FDE:Rbc}.
\end{theorem}

\begin{proof} Let $u^* \in H^{1-\beta,0}_{R,l}(0,1)$ be the unique solution to \eqref{RR:e1} as defined in Theorem \ref{thm:RR}. We introduce the auxiliary function $w_f$ as
\begin{equation}\label{RR_Char:e3}
w_f := {}_0 I^{\beta}_x u^*.
\end{equation}
Then $w_f \in H^1(0,1)$ is a weak solution to the weak formulation
\begin{equation}\label{RR_Char:e4}
\big(Dw_f, Dv \big)_{L^2(0,1)} = \langle f,v\rangle + a_1 v(1) - a_0 v(0), \quad \forall v \in H^1(0,1).
\end{equation}
Furthermore, $w_f$ satisfies the constraint
\begin{equation*} \int_0^1 w_f dx = \int_0^1 {}_0 I^{\beta}_x u^*(x) dx = 0. \end{equation*}
Lemma \ref{lem:Fred} shows that the bilinear form on the left-hand side of \eqref{RR_Char:e4} is coercive on $H^{1,0}(0,1) \times H^{1,0}(0,1)$ and so has a unique solution $w_f \in H^{1,0}(0,1)$. Any solution $w \in H^1(0,1)$ to \eqref{RR_Char:e4} can be expressed as $w = w_f + C$. Hence,  $w(x) - w(0) = w_f(x) - w_f(0)$ is uniquely determined for any weak solution $w$ to \eqref{RR_Char:e4}.
Furthermore, the function $u$ defined by \eqref{RR_Char:e1} is the general solution to problem \eqref{FDE:RL} and \eqref{FDE:Rbc}.
\end{proof}

\section{Two other reducible cases}
\setcounter{section}{7}\setcounter{equation}{0}

We study the Riemann--Liouville equation \eqref{FDE:RL} with the Caputo boundary condition \eqref{FDE:Cbc} and the conservative Caputo equation \eqref{FDE:Conserv} with the Riemann--Liouville boundary condition \eqref{FDE:Rbc}.

\begin{theorem}\label{thm:R_Cbc} Let $0 < \beta < 1$ and $f \in L^2(0,1)$ satisfy the constraint \eqref{ConCbc:e2}. Then the Riemann--Liouville equation \eqref{FDE:RL} with the Caputo boundary condition \eqref{FDE:Cbc} has a unique weak solution $u \in H^{1-\beta/2}_{C,l}(0,1) \cap H^{1-\beta/2}_{R,l}(0,1)$ which satisfies $u(0) = 0$. 
\end{theorem}

\begin{proof}  Theorem \ref{thm:ConCbc} concludes that the conservative Caputo equation \eqref{FDE:Conserv} with the Caputo boundary condition \eqref{FDE:Cbc} has a unique solution $u \in H^{1-\beta/2}_{C,l}(0,1)$ up to an arbitrary constant. By Theorem \ref{thm:embed}, $H^{1-\beta/2}_{C,l}(0,1) \hookrightarrow C[0,1]$, so $u(0)$ is well defined. Motivated by the identity \cite{Pod}
\begin{equation}\label{R_Cbc:e1}
{}^{R}_0D_x^{1-\beta/2}u(x) = \f{u(0)}{\Gamma(\beta)x^{1-\beta/2}} + ~{}^{C}_0D_x^{1-\beta/2}u(x)
\end{equation}
we choose a particular solution $\tilde{u} \in H^{1-\beta/2}_{C,l}(0,1)$ with $\tilde{u}(0) = 0$. Then, ${}^{R}_0D_x^{1-\beta/2} \tilde{u} =  {}^{C}_0D_x^{1-\beta/2}\tilde{u}$. This means that $\tilde{u} \in H^{1-\beta/2}_{R,l}(0,1) \cap H^{1-\beta/2}_{C,l}(0,1)$ is a weak solution to the Riemann--Liouville equation \eqref{FDE:RL} with the Caputo boundary condition \eqref{FDE:Cbc}. In other words, problem \eqref{FDE:RL} and \eqref{FDE:Cbc} has at least one weak solution $\tilde{u}$. 

Next, let $u$ be any weak solution to the problem. Then the difference $w := u - \tilde{u}$ satisfies the homogeneous Riemann--Liouville equation \eqref{FDE:RL} with the homogeneous Caputo boundary condition \eqref{FDE:Cbc}. We solve this equation to get ${}_0I_x^{\beta}w = C_0 x + C_1$. We apply ${}^{R}_0D_x^{\beta}$ to this equation and use Lemma \ref{lem:Aux3} and \eqref{FDE:e0} to find  
\begin{equation}\label{R_Cbc:e2}\begin{aligned}
&w = D {}^{}_0I_x^{1-\beta}\big(C_0 x + C_1 \big) = \f{C_0x^{1-\beta}}{\Gamma(2-\beta)} + \f{C_1}{\Gamma(1-\beta)x^{\beta}},\\
& Dw = \f{C_0}{\Gamma(1-\beta)x^\beta}+ \f{C_1}{\Gamma(-\beta) x^{1+\beta}}. \end{aligned} \end{equation}
In order for ${}^{C}_0D_x^{1-\beta}w$ to be defined, $C_1$ has to be chosen to be 0. Then ${}^{C}_0D_x^{1-\beta}w = C_0$. The homogeneous Caputo boundary condition concludes that $C_0 = 0$. Thus, $w \equiv 0$ and the uniqueness of the weak solution is proved.
\end{proof}

\begin{theorem}\label{thm:CaputoRbc} Let $0 < \beta < 1$, $\delta > 0$, and $f \in C^\delta[0,1]$ satisfy the constraint \eqref{ConCbc:e2}. Then the Caputo equation \eqref{FDE:Caputo} with the Riemann--Liouville boundary condition \eqref{FDE:Rbc} has a unique solution $u = {}^{C}_0 D^{1-\beta}_x w$, where $w \in C^{2+\delta}[0,1]$ with $w(0) = 0$. 
\end{theorem}

\begin{proof}
Theorem \ref{thm:RR} concludes that the Riemann--Liouville equation \eqref{FDE:RL} with the Riemann--Liouville boundary condition \eqref{FDE:Rbc} has a unique solution $\hat{u} \in H^{1-\beta,0}_{R,l}(0,1)$. Furthermore, $w := {}_0 I^{\beta}_x \hat{u}$ satisfies the classical Neumann boundary--value problem \eqref{RR_Char:e2}. The classical differential equation theory \cite{GilTru} concludes that $w \in C^{2+\delta}[0,1]$ is determined uniquely up to an arbitrary constant. That is, ${}_0 I^{\beta}_x \hat{u} = w + C$ with $w(0) = 0$. We apply ${}^{R}_0D_x^{\beta}$ to this equation to get
\[\hat{u} = D {}_0I_x^{1-\beta}(w+C) = D {}_0I_x^{1-\beta}w + \f{C}{\Gamma(1-\beta)x^{\beta}}.\] 
Clearly ${}^{R}_0D_x^{1-\beta}u = Dw$ for any $C \in \mathbb{R}$. But ${}^{C}_0D_x^{1-\beta}\hat{u}$
is defined only for $C = 0$. In this case.
\[ \begin{aligned} 
{}^{C}_0D_x^{1-\beta}\hat{u} & = {}_0I_x^{\beta}D D {}_0I_x^{1-\beta}w = {}_0I_x^{\beta}D  {}_0I_x^{1-\beta} D w  \\
& = {}_0I_x^{\beta}D \Big (\f{Dw(0)x^{1-\beta}}{\Gamma(2-\beta)} + \f1{\Gamma(2-\beta)} \int_0^x (x-s)^{1-\beta}D^2w(s)ds \Big) \\
& = {}_0I_x^{\beta} \Big (\f{Dw(0)}{\Gamma(1-\beta)x^\beta} + \f1{\Gamma(2-\beta)} \f{d}{dx}\int_0^x (x-s)^{1-\beta}D^2w(s)ds \Big) \\
& = Dw(0) + {}_0I_x^{\beta} \Bigl ( \f1{\Gamma(1-\beta)} \int_0^x (x-s)^{-\beta} D^2 w(s) ds \Big ) \\
& = Dw(0) + {}_0I_x^{\beta}  {}_0I_x^{1-\beta} D^2 w = Dw(0) + {}_0I^{-1}_x D^2 w \\
& = Dw(0) + {}_0I_x D^2 w = Dw.
\end{aligned} \] 
In short, for $C = 0$ we have ${}^{C}_0D_x^{1-\beta}\hat{u} = {}^{R}_0D_x^{1-\beta}\hat{u}$. In other words, $\hat{u}$ is a particular solution to the conservative Caputo FDE \eqref{FDE:Conserv} with the Riemann--Liouville boundary condition \eqref{FDE:Rbc}. That is, problem \eqref{FDE:Conserv} and \eqref{FDE:Rbc} has at least one solution.

Let $u$ be any solution to problem \eqref{FDE:Conserv} and \eqref{FDE:Rbc}. Then the difference $\hat{w} := u - \hat{u}$ satisfies the homogeneous conservative Caputo FDE \eqref{FDE:Conserv} with $f=0$ with the homogeneous Riemann--Liouville boundary condition \eqref{FDE:Rbc} with $a_0 = a_1 = 0$. Solving this problem yields
\[\begin{aligned}
&D\hat{w} = D {}_0I_x^{1-\beta}C_0 = D \Bigl (\f{C_0x^{1-\beta}}{\Gamma(2-\beta)} \Big), \quad
\hat{w} = \f{C_0x^{1-\beta}}{\Gamma(2-\beta)} + C_1, \\
& {}^{R}_0D_x^{1-\beta} \hat{w} = C_0 + \f{C_1}{\Gamma(\beta) x^{1-\beta}}. \end{aligned} \]
Enforcing the homogeneous Riemann--Liouville boundary condition \eqref{FDE:Rbc} reveals that $C_0 = C_1 = $0. Hence, $\hat{w} \equiv 0$ and the uniqueness of the solution is proved.
\end{proof}

\section{The remaining Neumann boundary--value problems}\label{Counter}
\setcounter{section}{8}\setcounter{equation}{0}

We show that the remaining combinations of the FDEs and the Neumann boundary conditions do not admit a solution in general.

\subsection{The Caputo FDE} \label{CounterCaputo}

Let $u$ be any solution to \eqref{FDE:Caputo} with $f = -1$. We apply ${}^{R}_0D_x^{\beta}$ to both sides of this equation and use Lemma \ref{lem:Aux3} and \eqref{FDE:e0} to obtain 
\begin{equation}\label{7A:e1} \begin{aligned}
&D^2 u = D {}_0I_x^{1-\beta} 1 = D \Bigl (\f{x^{1-\beta}}{\Gamma(2-\beta)} \Big) = \f{1}{\Gamma(1-\beta)x^{\beta}},\\
&Du = \f{x^{1-\beta}}{\Gamma(2-\beta)} + C_0, \quad u = \f{x^{2-\beta}}{\Gamma(3-\beta)} + C_0 x + C_1,\\
& {}^{C}_0D_x^{1-\beta}u = x + \f{C_0x^\beta}{\Gamma(1+\beta)}, \quad
{}^{R}_0D_x^{1-\beta}u = x + \f{C_0x^\beta}{\Gamma(1+\beta)} + \f{C_1}{\Gamma(\beta)x^{1-\beta}}.
\end{aligned} \end{equation}
In the current context, the constraint \eqref{Caputo:e2} reduces to $a_1 - a_0 = 1/\Gamma(2-\beta)$. With $C_0 = a_0$, $u$ satisfies the classical Neumann boundary condition \eqref{FDE:Nbc}. This agrees with Theorem \ref{thm:Caputo}. However, as long as $a_0 \neq 0$, $u$ cannot satisfy the Caputo boundary condition \eqref{FDE:Cbc} or the Riemann--Liouville boundary condition \eqref{FDE:Rbc}. In other words, the Caputo FDE \eqref{FDE:Caputo} with the boundary condition \eqref{FDE:Cbc} or \eqref{FDE:Rbc} does not admit a solution, in general.

\subsection{The conservative Caputo FDE} 

Let $u$ be any solution to \eqref{FDE:Conserv} with $f=-1$. We get ${}_0I_x^{\beta}Du = x + C_0$. We apply ${}^{R}_0D_x^{\beta}$ to this equation and use Lemma \ref{lem:Aux3} and \eqref{FDE:e0} to find  
\begin{equation}\label{7B:e1}\begin{aligned}
&Du = D {}_0I_x^{1-\beta}(x+C_0) = D \Bigl (\f{x^{2-\beta}}{\Gamma(3-\beta)}+\f{C_0x^{1-\beta}}{\Gamma(2-\beta)} \Big),\\
&u = \f{x^{2-\beta}}{\Gamma(3-\beta)}+\f{C_0x^{1-\beta}}{\Gamma(2-\beta)} + C_1,
\quad Du = \f{x^{1-\beta}}{\Gamma(2-\beta)}+\f{C_0}{\Gamma(1-\beta)x^\beta},\\
& {}^{C}_0D_x^{1-\beta} u = x + C_0, \quad {}^{R}_0D_x^{1-\beta} u = x + C_0 + \f{C_1}{\Gamma(\beta) x^{1-\beta}}. \end{aligned} \end{equation}
The constraint \eqref{ConCbc:e2} in Theorem \ref{thm:ConCbc} reduces to 
\begin{equation}\label{7B:e2} a_1 - a_0 = 1.\end{equation}
It is clear that with the choice of $C_0 = a_0$, $u$ satisfies the Caputo boundary condition \eqref{FDE:Cbc}. As proved in Theorem \ref{thm:ConCbc}. It is also clear that $u$ cannot satisfy the classical Neumann boundary condition \eqref{FDE:Nbc} as long as $a_0 \neq 0$. Thus, problem \eqref{FDE:Conserv} and \eqref{FDE:Nbc} does not admit a solution in general. Enforcing the Riemann--Liouville boundary condition \eqref{FDE:Rbc} at $x=0$ reveals that $C_0 = a_0 $ and $C_1 = 0$, and a unique solution of \begin{equation}\label{7B:e3} u = \f{x^{2-\beta}}{\Gamma(3-\beta)} + \f{a_0x^{1-\beta}}{\Gamma(2-\beta)} \end{equation}
is obtained, as proved in Theorem \ref{thm:CaputoRbc}.

\subsection{The Riemann--Liouville FDE} 

Let $u$ be any solution to \eqref{FDE:RL} with $f=-1$. We get ${}_0I_x^{\beta}u = x^2/2 + C_0 x + C_1$. We apply ${}^{R}_0D_x^{\beta}$ to this equation and use Lemma \ref{lem:Aux3} and \eqref{FDE:e0} to find  
\begin{equation}\label{7C:e1}\begin{array}{rl}
\ds u & \ds = D {}_0I_x^{1-\beta}\Big(\f{x^2}2 + C_0 x + C_1 \Big) \\
&\ds = \f{x^{2-\beta}}{\Gamma(3-\beta)} + \f{C_0x^{1-\beta}}{\Gamma(2-\beta)} + \f{C_1}{\Gamma(1-\beta)x^{\beta}},\\
{}^{R}_0D_x^{1-\beta}u & = x + C_0, \\
Du & \ds = \f{x^{1-\beta}}{\Gamma(2-\beta)}+\f{C_0}{\Gamma(1-\beta)x^\beta}+ \f{C_1}{\Gamma(\beta) x^{1+\beta}}.
\end{array} \end{equation}
The constraint \eqref{7B:e2} is still true. It is clear that with the choice of $C_0 = a_0$, $u$ satisfies the Caputo boundary condition \eqref{FDE:Rbc}. That is, problem \eqref{FDE:RL} and \eqref{FDE:Rbc} has a solution as proved in Theorem \ref{thm:RR}. On the other hand, It is also clear that $u$ cannot satisfy the classical Neumann boundary condition \eqref{FDE:Nbc} as long as $a_0 \neq 0$. Thus, problem \eqref{FDE:RL} and \eqref{FDE:Nbc} does not admit a solution in general.

We note from the expression of $Du$ in \eqref{7C:e1} that ${}^{C}_0D_x^{1-\beta}u$ is defined if and only if $C_1 = 0$. In this case, ${}^{C}_0D_x^{1-\beta}u = x + C_0$. The boundary condition \eqref{FDE:Cbc} yields $C_0 = a_0$. The constraint \eqref{7B:e2} is automatically satisfied. Thus, a unique solution of the form \eqref{7B:e3} is obtained, as proved in Theorem \ref{thm:R_Cbc}. 

We end this section by showing that the nonsolvability of the four remaining combinations of FDEs with the Neumann boundary conditions as shown by the counterexample holds in general. We take the conservative Caputo FDE \eqref{FDE:Conserv} and the classical Neumann boundary condition \eqref{FDE:Nbc} as an example to demonstrate the idea. Let $0 < \beta < 1$ and $f \in H^1(0,1)$. Then the conservative Caputo FDE \eqref{FDE:Conserv} with the classical Neumann boundary condition \eqref{FDE:Nbc} is solvable if and only if $a_0 = 0$. 

In fact, let $u$ be the solution to this problem. Then $w := Du$ would satisfy the following inhomogeneous Dirichlet boundary-value problem of the Riemann--Liouville FDE
\begin{equation}\label{ConNbc:e1}\left\{\begin{aligned}
& -  \bigl({}^{R}_0D_x^{2-\beta} w \bigr) = Df(x), \quad  x \in (0,1), \\
&  v(0) = a_0, \quad v(1) = a_1.
\end{aligned} \right.
\end{equation}
Let $w_b := a_0 (1-x) + a_1 x$ and $w_f := w - w_b$. Then we can rewrite \eqref{ConNbc:e1}  as 
\begin{equation}\label{ConNbc:e2}\left\{\begin{aligned}
& -  \bigl({}^{R}_0D_x^{2-\beta} w_f \bigr) = Df(x) + f_b , \quad  x \in (0,1), \\
&  v(0) = a_0, \quad v(1) = a_1.
\end{aligned} \right.
\end{equation}
where $f_b$ is defined by
\begin{equation}\label{GAL_RL:e4}
f_b(x) := \f{(a_1 - a_0) (1+\beta) }{\Gamma(\beta)x^{1-\beta}} + \frac{a_0(\beta-1)}{\Gamma(\beta)x^{2-\beta}}.
\end{equation}
It was proved in \cite{ErvRoo} that problem \eqref{ConNbc:e2} has a unique weak solution $w \in H^{1-\beta/2}_0(0,1)$ if and only if the right-hand side of the equation is in $H^{-(1-\beta/2)}(0,1)$. Note that the first term on the right-hand side of \eqref{ConNbc:e2} is in $L^1(0,1) \hookrightarrow H^{-(1-\beta/2)}(0,1)$. The second term is a distributional derivative of a constant multiple of the first term, which is not in $H^{-(1-\beta/2)}(0,1)$. Hence, this problem has a unique solution if and only if $a_0 = 0$.

\section*{Acknowledgements}


This work was supported in part by the OSD/ARO MURI Grant W911NF-15-1-0562, by the National Science Foundation under Grants  EAR-0934747, DMS-1216923 and DMS-1620194, and by the National Natural Science Foundation of China under Grants 11201485, 91130010, 11471194 and 11571115.




 \bigskip \smallskip

 \it

 \noindent

$^1$ Department of Mathematics, University of South Carolina\\
Columbia, South Carolina 29208, USA \\[4pt]
  e-mail: hwang@math.sc.edu

$^2$ Department of Mathematics, East China Normal University\\
 Shanghai, 200241, China \\[4pt]
  e-mail: dpyang@math.ecnu.edu.cn

\end{document}